\documentclass[11pt]{article} 
\usepackage{float}
\usepackage{fourier} 
\usepackage{amsmath,amssymb,amsthm}
\usepackage{latexsym}
\usepackage{fullpage} 
\usepackage{subcaption} 
\usepackage{graphicx} 
\usepackage{authblk}

\usepackage{url} 
\usepackage[round]{natbib} 
\usepackage{color} 
\usepackage{physymb}

\usepackage[boxruled]{algorithm2e} 
\usepackage[colorlinks=true,citecolor=red,linkcolor=blue]{hyperref} 
\usepackage{natbib} 
\newtheorem{theorem}{Theorem}
\newtheorem{lemma}{Lemma}

\theoremstyle{remark}
\newtheorem{remark}{Remark}

\theoremstyle{definition}

\newcommand{\mbb}[1]{\mathbb{#1}}  

\begin{document} 
\title{On Detection and Structural Reconstruction of Small-World Random Networks}

\date{}
\author[1]{T. Tony Cai}
\author[1]{Tengyuan Liang}
\author[1]{Alexander Rakhlin}

\affil[1]{Department of Statistics, The Wharton School, University of Pennsylvania} 

\renewcommand\Authands{ and }

\maketitle 
\begin{abstract}
In this paper, we study detection and fast reconstruction of the celebrated Watts-Strogatz (WS) small-world random graph model  \citep{watts1998collective} which aims to describe real-world complex networks that exhibit both high clustering and short average length properties. The WS model with neighborhood size $k$ and rewiring probability probability $\beta$ can be viewed as a continuous interpolation between a deterministic ring lattice graph and the Erd\H{o}s-R\'{e}nyi  random graph. We study both the computational and statistical aspects of detecting the deterministic ring lattice structure (or local geographical links, strong ties) in the presence of random connections (or long range links, weak ties), and for its recovery. The phase diagram in terms of $(k,\beta)$ is partitioned into several regions according to the difficulty of the problem. We propose distinct methods for the various regions.

\end{abstract}

\section{Introduction} 
The ``small-world'' phenomenon aims to describe real-world complex networks that exhibit both high clustering and short average length properties. While most of the pairs of nodes are not friends, any node can be reached from another in a small number of hops. The Watts-Strogatz (WS) model, introduced in \citep{watts1998collective, newman1999scaling}, is a popular generative model for networks that exhibit the small-world phenomenon. The WS model interpolates between the two extremes---the regular lattice graph on the one hand, and the random graph on the other. Considerable effort has been spent on studying the asymptotic statistical behavior (degree distribution, average path length, clustering coefficient, etc.) and the empirical performance of the WS model \citep{watts1999small, amaral2000classes,barrat2000properties,latora2001efficient, van2009random}. Successful applications of the WS model have been found in a range of disciplines, such as psychology \citep{milgram1967small},  epidemiology \citep{moore2000epidemics}, medicine and health \citep{stam2007small}, to name a few. 
In one of the first algorithmic studies of the small-world networks, \cite{kleinberg2000small} investigated the theoretical difficulty of finding the shortest path between any two nodes when one is restricted to use local algorithms, and further related the small-world notion to long range percolation on graphs \citep{benjamini2000diameter,coppersmith2002diameter}. %Many more theoretical questions are yet to be answered. 
The focus of the present paper is on statistical and computational aspects of the detection and recovery problems.

Given a network, the first statistical and computational challenge is to detect whether it enjoys the small world property, or whether the observation may be explained by the Erd\H{o}s-R\'{e}nyi random graph (the null hypothesis). The second question is concerned with the reconstruction of the neighborhood  structure if the network does exhibit the small world phenomenon. In the language of social network analysis, the detection problem corresponds to detecting the existence of local geographical links (or close friend connections, strong ties) in the presence of long range links (or random connections, weak ties). The reconstruction problem corresponds to distinguishing between these local links and long range links. The problem is statistically and computationally difficult due to the high-dimensional unobserved latent variable---the permutation matrix--- which blurs the natural ordering of the ring structure. 

Let us parametrize the WS model in the following way: the number of nodes is denoted by $n$, the neighborhood size by $k$, and the rewiring probability by $\beta$. Provided the adjacency matrix $A \in \mathbb{R}^{n\times n}$, we are interested in identifying the choices of $(n,k,\beta)$ when detection and reconstruction of the small world random graph is possible. Specifically, we focus on the following two questions. 

\medskip
\noindent {\bf Detection}~~ Given the adjacency matrix $A$ up to a permutation, when (in terms of $n,k,\beta$) and how (in terms of procedures) can one statistically distinguish whether it is a small world graph ($\beta < 1$), or a usual random graph with matching degree ($\beta = 1$). What if we restrict our attention to computationally efficient procedures? %This questions correspond to detect the presence of deterministic neighborhood structure. 

\medskip
\noindent {\bf Reconstruction}~~ Once the presence of the neighborhood structure is confirmed, when (in terms of $n,k,\beta$) and how (in terms of procedures) can one estimate the deterministic neighborhood structure? If one only aims to estimate the structure consistently (asymptotically correct), are there computationally efficient procedures, and what are their limitations?

\medskip
We address the above questions by presenting a phase diagram in  Figure~\ref{fig:phase.tran}. The phase diagram divides the parameter space into four disjoint regions according to the difficulty of the problem. We propose distinct methods for the regions where solutions are possible.

% \medskip
% \noindent {\bf Detection}: what is the relationship for $(n,k,\beta)$ in order to distinguish between small world graph (with ring lattice structure) and pure random graph with matching degree,
% when you observe the actual adjacency matrix $A$ up to a permutation. Furthermore, what is the optimal test statistics? What if we restrict to computationally efficient procedures?

% \medskip
% \noindent {\bf Reconstruction}: provided the observed adjacency matrix $A$, under what regime for $(n,k,\beta)$ it is impossible to reconstruct the structure information theoretically? Under what regime there exists a computationally efficient algorithm that recovers the regular lattice/ring structure?

\subsection{Why Small World Graph?}
Finding and analyzing the appropriate statistical models for real-world complex networks is one of the main themes in network science. Many real empirical networks---for example, internet architecture, social networks, and biochemical pathways---exhibit two features simultaneously: high clustering among individual nodes and short distance between any two nodes. Consider the local tree rooted at a person. The high clustering property suggests prevalent existence of triadic closure, which significantly reduces the number of reachable people within a certain depth (in contrast to the regular tree case where this number grows exponentially with the depth), contradicting the short average length property. In a pathbreaking paper, \cite{watts1998collective} provided a  mathematical model that resolves the above seemingly contradictory notions.
%, which later stimulates the well-known theory  ``six degrees of separation''. 
The solution is surprisingly simple --- interpolating between structural ring lattice graph and a random graph. The ring lattice provides the strong ties (i.e., homophily, connection to people who are similar to us) and triadic closure, while the random graph generates the weak ties (connection to people who are otherwise far-away), preserving the local-regular-branching-tree-like structure that induces short paths between pairs. %The model has been influential in a number of disciplines, see \citep{latora2002boston, achard2006resilient, stam2007small, cassar2007coordination}. 

Given the small world model, it is natural to ask the statistical question of distinguishing the local links (geographical) and long range links (non-geographical) based on the observed graph.

\subsection{Rewiring Model} 
Let us now define the WS model. Consider a ring lattice with $n$ nodes, where each node is connected with its $k$ nearest neighbors ($k/2$ on the left and $k/2$ on the right, $k$ even for convenience). The rewiring process contains two procedures: erase and reconnect. First, erase each currently connected edge with probability $\beta$, independently. Next, reconnect each edge pair with probability $\beta \frac{k}{n-1}$, allowing multiplicity.\footnote{The original rewiring process in \cite{watts1998collective} does not allow multiplicity; however, for the simplicity of technical analysis, we focus on reconnection allowing multiplicity. These two rewiring processes are asymptotically equivalent. } The observed symmetric adjacency matrix $A \in \{ 0 ,1 \}^{n \times n}$ has the following structure under some unobserved permutation matrix $P_\pi \in \{ 0, 1 \}^{n \times n}$. For $1\leq i <j \leq n$,
\begin{equation*}
	[P_\pi A P_\pi^T]_{ij} = 1 ~~\left\{ 
	\begin{array}{ll}
		\text{w.p.}~~ 1- \beta (1-\beta\frac{k}{n-1}), & \text{if}~ 0<|i-j| \leq \frac{k}{2} \mod n-1-\frac{k}{2} \\
		\text{w.p.}~~ \beta \frac{k}{n-1}, & \text{otherwise} 
	\end{array}
	\right. 
\end{equation*}
and entries are independent of each other. Equivalently, we have for $1\leq i <j \leq n$ 
\begin{align}
	 \label{eq:model}
	 A_{ij} = \kappa \left([P_\pi B P_\pi^T]_{ij}\right),
\end{align}
where $\kappa(\cdot)$ is the entry-wise i.i.d. Markov channel, 
\begin{align*}
	\kappa(0) &\sim {\sf Bernoulli}\left(\beta \frac{k}{n-1}\right), \\
	\kappa(1) &\sim {\sf Bernoulli}\left(1- \beta (1-\beta\frac{k}{n-1})\right),
\end{align*}
and $B \in \{ 0,1 \}^{n \times n}$ indicates the support of the structural ring lattice
\begin{equation}
	\label{eq:structure.mat}
	B_{ij} = ~~\left\{ 
	\begin{array}{ll}
		1, & \text{if}~ 0<|i-j| \leq \frac{k}{2} \mod n-1-\frac{k}{2} \\
		0, & \text{otherwise} 
	\end{array}
	\right.  .
\end{equation}

We denote by ${\sf WS}(n, k, \beta)$ the distribution of the random graph generated from the rewiring model, and denote by ${\sf ER}(n, \frac{k}{n-1})$ the Erd\H{o}s-R\'{e}nyi random graph distribution (with matching average degree $k$). Remark that if $\beta = 1$, the small world graph ${\sf WS}(n, k, \beta)$ reduces to ${\sf ER}(n, \frac{k}{n-1})$, with no neighborhood structure. In contrast, if $\beta=0$, the small world graph ${\sf WS}(n, k, \beta)$ corresponds to the deterministic ring lattice, without random connections. 
We focus on the dependence of the gap $1-\beta = o(1)$ on $n$ and $k$, such that distinguishing between ${\sf WS}(n, k, \beta)$ and ${\sf ER}(n, \frac{k}{n-1})$ or reconstructing the ring lattice structure is statistically and computationally possible.

\subsection{Summary of Results} 
\label{sec:sum.result}
The main theoretical and algorithmic results are summarized in this section. We first introduce several regions in terms of $(n, k, \beta)$,  according to the difficulty of the problem instance, and then we present the results using the phase diagram in  Figure~\ref{fig:phase.tran}. Except for the impossible region, we will introduce different algorithms with distinct computational properties.

\medskip

\noindent {\bf Impossible region}: $1 - \beta \prec \sqrt{\frac{\log n}{n}} \vee \frac{\log n}{k}$. Inside this region no multiple testing procedure (regardless of computational budget) can succeed in distinguishing among the class of models including all of ${\sf WS}(n, k, \beta)$ and ${\sf ER}(n, \frac{k}{n-1})$ with vanishing error.

\medskip
 
\noindent {\bf  Hard region}: $\sqrt{\frac{\log n}{n}} \vee \frac{\log n}{k} \preceq 1 - \beta \prec  \sqrt{ \frac{1}{k} } \vee  \frac{\sqrt{\log n}}{k}$. It is possible to detect between ${\sf WS}(n, k, \beta)$ and ${\sf ER}(n, \frac{k}{n-1})$ statistically with vanishing error; however the evaluation of the test statistic \eqref{eq:max.lik.test} (below) requires exponential time complexity (to the best of our knowledge).

\medskip

\noindent {\bf Easy region}: $ \sqrt{ \frac{1}{k} } \vee  \frac{\sqrt{\log n}}{k} \preceq 1 - \beta \preceq  \sqrt{ \sqrt{\frac{\log n}{n}} \vee \frac{\log n}{k} }$. There exists an efficient spectral test that can distinguish between the small world random graph ${\sf WS}(n, k, \beta)$ and the  Erd\H{o}s-R\'{e}nyi graph ${\sf ER}(n, \frac{k}{n-1})$ in near linear time (in the matrix size). 

\medskip

\noindent {\bf Reconstructable region}: $\sqrt{ \sqrt{\frac{\log n}{n}} \vee \frac{\log n}{k} }  \prec 1 - \beta \preceq 1$. In this region, not only is it possible to detect the existence of the lattice structure in a small-world graph, but it is also possible  computationally to consistently estimate/reconstruct the neighborhood structure via a novel correlation thresholding procedure. 

\medskip

The following phase diagram provides an intuitive illustration of the above theoretical results. If we parametrize $k \asymp n^{x}, 0 <x<1$ and $1 - \beta \asymp n^{-y}, 0<y<1$, each point $(x,y) \in [0,1]^2$ corresponds to a particular problem instance with parameter bundle $(n, k=n^x, \beta = 1 - n^{-y})$. According to the location of $(x,y)$, the difficulty of the problem changes; for instance, the larger the $x$ and the smaller the $y$ is, the easier the problem becomes. The various above regions (plotted in $[0,1]^2$) are: impossible region ({\color{red} red} region \uppercase\expandafter{\romannumeral1}), hard region ({\color{blue} blue} region \uppercase\expandafter{\romannumeral2}), easy region ({\color{green} green} region \uppercase\expandafter{\romannumeral3}), reconstructable region ({\color{cyan} cyan} region \uppercase\expandafter{\romannumeral4}). 
\begin{figure}[H] 
	\centering 
	\includegraphics[width=3in]{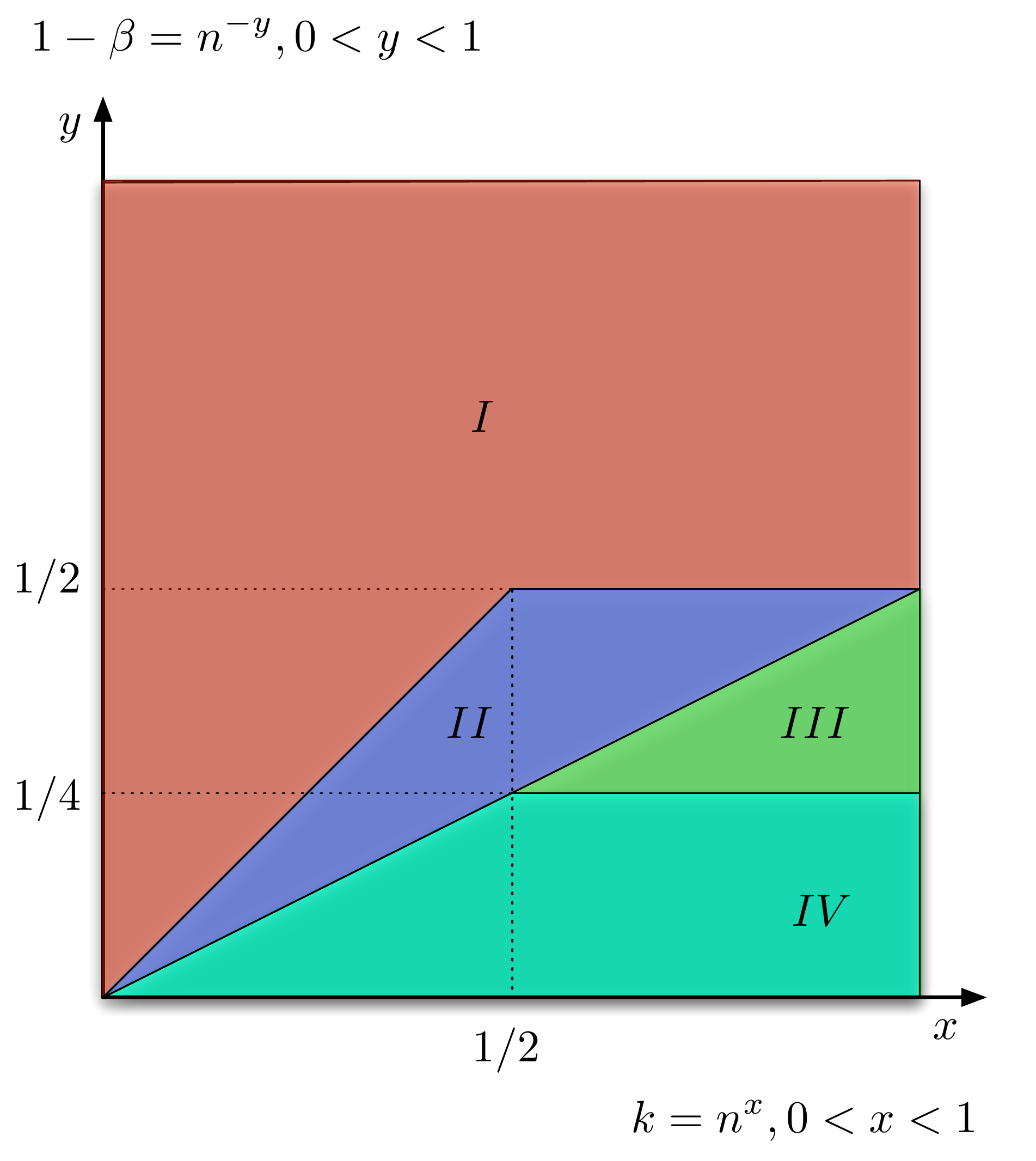} 
	\caption{Phase diagram for small-world network: impossible region ({\color{red} red} region \uppercase\expandafter{\romannumeral1}), hard region ({\color{blue} blue} region \uppercase\expandafter{\romannumeral2}), easy region ({\color{green} green} region \uppercase\expandafter{\romannumeral3}), and reconstructable region ({\color{cyan} cyan} region \uppercase\expandafter{\romannumeral4}).}
	% \caption{Difficulty Region: Green region is computationally easy region, where detection is easy via polynomial time algorithm; Red region is information theoretically impossible region, where detection is statistically hard; Blue region is when the problem is easy for exponential time algorithm, but hard for polynomial time algorithm. }
	\label{fig:phase.tran}
\end{figure}

 \subsection{Notation}
 $A, B, Z \in \mathbb{R}^{n\times n}$ denote symmetric matrices: $A$ is the adjacency matrix, $B$ is the structural signal matrix as in Equation~\eqref{eq:structure.mat}, and $Z=A - \mathbb{E}A$ is the noise matrix. We denote the matrix of all ones by $J$. Notations $\preceq$, $\succeq$, $\prec$, $\succ$ denote the asymptotic order: $a(n) \preceq b(n)$ if and only if $\limsup\limits_{n\rightarrow \infty} \frac{a(n)}{b(n)} \leq c$, with some constant $c>0$, $a(n) \prec b(n)$ if and only if $\limsup\limits_{n\rightarrow \infty} \frac{a(n)}{b(n)} = 0$. $C, C'>0$ are universal constants that may change from line to line. For a symmetric matrix $A$, $\lambda_i(A), 1\leq i\leq n$ denotes the ranked eigenvalues, in a decreasing order. The inner-product $\langle A,B \rangle = {\rm tr}(A^T B)$ overloads the usual Euclidian inner-product and matrix inner-product. For any integer $n$, $[n]:=\{0, 1, \ldots, n-1\}$ denotes the index set. Denote the permutation in symmetric group $\pi \in S_n$ and its associated matrix form as $P_\pi$.

For a graph $G(V,E)$ generated from the Watts-Strogatz model ${\sf WS}(n, k, \beta)$ with associated permutation $\pi$, for each node $v_i \in V, 1\leq i\leq |V|$, we denote $$\mathcal{N}(v_i):= \left\{ v_j : 0<|\pi^{-1}(i)-\pi^{-1}(j)| \leq \frac{k}{2} \mod n-1-\frac{k}{2} \right\}$$ as the ring neighborhood nodes with respect to node $v_i$, before the permutation $\pi$ applied.

\subsection{Organization of the Paper}
	The following sections are dedicated to the theoretical justification of the various regions in Section~\ref{sec:sum.result}. Specifically, Section~\ref{sec:region.1} establishes the boundary for the impossible region \uppercase\expandafter{\romannumeral1}, where the problem is impossible to solve information theoretically. We contrast the hard region \uppercase\expandafter{\romannumeral2} with the regions \uppercase\expandafter{\romannumeral3} and \uppercase\expandafter{\romannumeral4} in Section~\ref{sec:region.2.3}; here, the difference arises in   statistical and computational aspects of detecting the strong tie structure inside random graph. Section~\ref{sec:region.4} studies a  correlation thresholding algorithm that reconstructs the neighborhood structure consistently when the parameters lie within the reconstructable region \uppercase\expandafter{\romannumeral4}. We also study a spectral ordering algorithm which succeeds in reconstruction in a part of region \uppercase\expandafter{\romannumeral3}. Whether the remaining part of region \uppercase\expandafter{\romannumeral3} admits a recovery procedure is an open problem. Additional further directions are listed in Section~\ref{sec:discuss}.

\section{The Impossible Region: Lower Bounds}
\label{sec:region.1}

We start with an information theoretic result that describes the difficulty of distinguishing among a class of models. The following Theorem~\ref{thm:low.bound} characterizes the impossible region, as in Section~\ref{sec:sum.result}, in the language of minimax multiple testing error. The proof is postponed to Section~\ref{sec:proof}. 

\begin{theorem}[Impossible Region]
	\label{thm:low.bound}
	Consider the following statistical models: $\mathcal{P}_0$ denotes the probability measure of the Erd\H{o}s-R\'{e}nyi random graph ${\sf ER}(n, \frac{k}{n-1})$, and $\mathcal{P}_\pi, \pi \in S_{n-1}$ denote the probability measures of the Watts-Strogatz small-world graph ${\sf WS}(n, k, \beta)$ as in Equation~\eqref{eq:model} with different permutations $\pi$. Consider any selector $\phi: \{0,1\}^{n\times n} \rightarrow S_{n-1} \cup \{ 0\}$ that maps from the adjacency matrix $A \in \{0,1\}^{n\times n}$ to a decision in $S_{n-1} \cup \{ 0\}$. Then for any fixed $0<\alpha<1/8$, the following lower bound on multiple testing error holds: 
	\begin{align*}
		\varliminf_{n\rightarrow \infty} ~ \min_{\phi} ~\max~\left\{ \mathcal{P}_0(\phi \neq 0),~ \frac{1}{(n-1)!} \sum_{\pi \in S_{n-1}} \mathcal{P}_\pi(\phi \neq \pi) \right\} \geq % \varliminf_{n\rightarrow \infty} ~ \frac{\sqrt{(n-1)!}}{1+\sqrt{(n-1)!}} \left( 1- 2\alpha - \sqrt{\frac{2\alpha}{\log (n-1)!}} \right)= 
		1 -2\alpha,
	\end{align*}
	when the parameters satisfy
	$$
	1-\beta \leq C_\alpha  \cdot \sqrt{\frac{\log n}{n}} \quad \text{or} \quad 1-\beta \leq C_\alpha' \cdot \frac{\log n}{k} \cdot \frac{1}{\log \frac{n \log n}{k^2}},
	$$
	with constants $C_\alpha, C_\alpha'$ that only depend on  $\alpha$.
		In other words, if
	$$
	1-\beta \prec \sqrt{\frac{\log n}{n}}\vee \frac{\log n}{k},
	$$
	no multiple testing procedure can succeed in distinguishing, with vanishing error, the class of models containing all of ~${\sf WS}(n, k, \beta)$ and ~${\sf ER}(n, \frac{k}{n-1})$.
\end{theorem}

	The missing latent random variable, the permutation matrix $P_\pi$, is the object we are interested in recovering. A permutation matrix $P_\pi$ induces a certain distribution on the adjacency matrix $A$. Thus the parameter space of interest, including models~${\sf WS}(n, k, \beta)$ and~${\sf ER}(n, \frac{k}{n-1})$, is of cardinality $(n-1)!+1$. Based on the observed adjacency matrix, distinguishing among the models including all of~${\sf WS}(n, k, \beta)$ and~${\sf ER}(n, \frac{k}{n-1})$ is equivalent to a multiple testing problem. The impossible region characterizes the information theoretic difficulty of this reconstruction problem by establishing the condition when minimax testing error does not vanish as $n, k(n) \rightarrow \infty$. 
	
	The ``high dimensional'' nature of this problem is mainly driven by the unknown permutation matrix, and this latent structure  introduces difficulty both statistically and computationally. Statistically, via Le Cam's method, one can build a distance metric on permutation matrices using the distance between the corresponding measures (measures on adjacency matrices induced by the permutation structure). In order to characterize the intrinsic difficulty of estimating the permutation structure, one needs to understand the richness of the set of permutation matrices within certain distance to one particular element, a combinatorial task. Computationally, the combinatorial nature of the problem makes the ``naive'' approach computationally intensive. %Smarter efficient approximation algorithms typically require case by case investigations.

\section{Hard v.s. Easy Regions: Detection Statistics}
\label{sec:region.2.3}

This section studies the hard and easy regions in Section~\ref{sec:sum.result}. First, we propose a near optimal test, the {\bf maximum likelihood test}, that detects the ring structure above the information boundary derived in Theorem~\ref{thm:low.bound}. However, the evaluation of the maximum likelihood test requires $\mathcal{O}(n^n)$ time complexity. The maximum likelihood test succeeds outside of region \uppercase\expandafter{\romannumeral1}, and, in particular, succeeds (statistically) in the hard region \uppercase\expandafter{\romannumeral2}. We then propose another efficient test, the {\bf spectral test}, that detects the ring structure in time $\mathcal{O^*}(n^2)$ via the power method. The method succeeds in regions \uppercase\expandafter{\romannumeral3} and \uppercase\expandafter{\romannumeral4}.

Theorem~\ref{thm:easy.hard.region} below combines the results of  Lemma~\ref{lma:maxlik.test} and Lemma~\ref{lma:spec.test}.

\begin{theorem}[Detection: Easy and Hard Boundaries]
	\label{thm:easy.hard.region}
	Consider the following statistical models: $\mathcal{P}_0$ denotes the distribution of the Erd\H{o}s-R\'{e}nyi random graph ${\sf ER}(n, \frac{k}{n-1})$, and $\mathcal{P}_\pi, \pi \in S_{n-1}$ denote distributions of the Watts-Strogatz small-world graph ${\sf WS}(n, k, \beta)$ with hidden permutation $\pi$. Consider any selector $\phi: \{0,1\}^{n\times n} \rightarrow \{ 0, 1\}$ that maps an adjacency matrix to a binary decision (detection decision).
	
		We say that minimax detection for the small-world random model is possible when
		\begin{align}
			\label{eq:minimax.detect}
		\lim_{n\rightarrow \infty}~\min_{\phi}~ \max~\left\{\mathcal{P}_0(\phi \neq 0), \frac{1}{(n-1)!} \sum_{\pi \in S_{n-1}} \mathcal{P}_\pi(\phi \neq 1)\right\} = 0.
		\end{align}
		If the parameter $(n,k,\beta)$ satisfies
		$$
		\text{hard boundary}:\quad 1-\beta \succeq \sqrt{\frac{\log n}{n}} \vee \frac{\log n}{k},
		$$
		minimax detection is possible, and an exponential time {\bf maximum likelihood test} \eqref{eq:max.lik.test} ensures  \eqref{eq:minimax.detect}.
		If, in addition, the parameter $(n,k,\beta)$ satisfies 
		$$
		\text{easy boundary}:\quad1 - \beta \succeq  \sqrt{ \frac{1}{k} } \vee  \frac{\sqrt{\log n}}{k},
		$$
		then a near-linear time {\bf spectral test} \eqref{eq:spec.test} ensures \eqref{eq:minimax.detect}.
\end{theorem}

Proof of Theorem~\ref{thm:easy.hard.region} consists of two parts, which will be addressed in the following two sections, respectively. 

\subsection{Maximum Likelihood Test} 
\label{sec:max.lik.test}
Consider the test statistic $T_1$ as the objective value of the following optimization 
\begin{align}
	\label{eq:max.lik}
	T_1(A):=\max_{P_\pi}~ \langle P_\pi B P_\pi^T, A \rangle,
\end{align}  
where $P_\pi \in \{ 0,1\}^{n\times n}$ is taken over all permutation matrices and $A$ is the observed adjacency matrix. The {\bf maximum likelihood test} $\phi_1 : A \rightarrow \{0,1\}$ based on $T_1$ by
\begin{equation}
	\label{eq:max.lik.test}
	\phi_1(A) = \left\{
	\begin{array}{cc}
		1 & \text{if}~T_1(A) \geq \frac{k}{n-1} nk + 2\sqrt{\frac{k}{n-1} nk  \cdot \log n!} + \frac{2}{3} \cdot \log n!\\
		0 & \text{o.w.}
	\end{array} \right.
\end{equation} 

The threshold is chosen as the rate $k^2 + \mathcal{O}\left(\sqrt{k^2 n \log \frac{n}{e}} \vee n \log \frac{n}{e}\right)$~: if the objective value is of a greater order, then we believe the graph is generated from the small-world rewiring process with strong ties; otherwise we we cannot reject the null, the random graph model with only weak ties. 

\begin{lemma}[Guarantee for Maximum Likelihood Test]
	\label{lma:maxlik.test}
	The maximum likelihood test $\phi_1$ in Equation~ \eqref{eq:max.lik.test} succeeds in detecting the small world random structure when $$ 1-\beta \succeq \sqrt{\frac{\log n}{n}} \vee \frac{\log n}{k},$$
	in the sense that
	\begin{align*}
		\lim_{n, k(n) \rightarrow \infty} ~\max~\left\{ \mathcal{P}_0(\phi_1 \neq 0),~ \frac{1}{(n-1)!} \sum_{\pi \in S_{n-1}} \mathcal{P}_\pi(\phi_1 \neq 1) \right\} = 0.
	\end{align*}
\end{lemma}

\begin{remark}
	Lemma~\ref{lma:maxlik.test} can be viewed as the condition on the signal and noise separation. By solving the combinatorial optimization problem, the test statistics aggregates the signal that separates from the noise the most. An interesting open problem is, if we solve a relaxed version of the combinatorial optimization problem~\eqref{eq:max.lik} within polynomial time complexity $\phi_1^{\sf rel}$, how much stronger the condition on  $1-\beta$ needs to be to ensure power.
\end{remark}

\subsection{Spectral Test} 
\label{sec:spec.test}

For the spectral test, we calculate the second largest eigenvalue of the adjacency matrix $A$ as the test statistic
\begin{align}
T_2(A) := \lambda_2 (A).	
\end{align}
The {\bf spectral test} $\phi_2 : A \rightarrow \{0,1\}$ is
\begin{equation}
	\label{eq:spec.test}
	\phi_2(A) = \left\{
	\begin{array}{cc}
		1 & \text{if}~T_2(A) \succeq \sqrt{k} \vee \sqrt{\log n}\\
		0 & \text{o.w.}
	\end{array} \right.
\end{equation} 
Namely, if $\lambda_2(A)$ passes a certain threshold, we classify the graph as a small-world graph. Evaluation of \eqref{eq:spec.test} only requires near-linear time $\mathcal{O}^*(n^2)$.

\begin{lemma}[Guarantee for Spectral Test]
	\label{lma:spec.test}
	The second eigenvalue test $\phi_2$ in Equation~\eqref{eq:spec.test} satisfies
	\begin{align*}
		\lim_{n, k(n) \rightarrow \infty} ~\max~\left\{ \mathcal{P}_0(\phi_2 \neq 0),~ \frac{1}{(n-1)!} \sum_{\pi \in S_{n-1}} \mathcal{P}_\pi(\phi_2 \neq 1) \right\} = 0
	\end{align*}
	whenever 
	\begin{align*}
		1 - \beta \succeq  \sqrt{ \frac{1}{k} } \vee  \frac{\sqrt{\log n}}{k}.
	\end{align*}
\end{lemma}

	The main idea behind Lemma~\ref{lma:spec.test} is as follows. Let us look at the expectation of the adjacency matrix, 
	\begin{align*}
		\mathbb{E} A = (1- \beta)(1-\beta\frac{k}{n-1})\cdot  P_\pi^T B P_\pi + \beta\frac{k}{n-1} \cdot (J-I),
	\end{align*}
	where $J$ is the matrix of all ones. The main structure matrix $P_\pi^T B P_\pi$ is a permuted version of the \emph{circulant matrix} (see e.g. \citep{gray2006toeplitz}). The spectrum of the circulant matrix $B$ is highly structured, and is of distinct nature in comparison to the noise matrix $A - \mathbb{E} A$.

\section{Reconstructable Region: Fast Structural Reconstruction}
\label{sec:region.4}

In this section, we discuss reconstruction of the ring structure in the Watts-Strogatz model. We show that in the reconstructable region (region \uppercase\expandafter{\romannumeral4} in Figure~\ref{fig:phase.tran}), a {\bf correlation thresholding procedure} succeed in reconstructing the ring neighborhood structure. As a by-product, once the neighborhood structure is known, one can distinguish between random edges and neighborhood edges for each node. A natural question is whether there is another algorithm that can work in a region (beyond region \uppercase\expandafter{\romannumeral4}) where correlation thresholding fails. We show that in a certain regime with large $k$, a {\bf spectral ordering procedure} outperforms the correlation thresholding procedure and succeeds in  parts of regions \uppercase\expandafter{\romannumeral3} and \uppercase\expandafter{\romannumeral4} (as depicted in Figure~\ref{fig:recon} below).

\subsection{Correlation Thresholding}

Consider the following {\bf correlation thresholding procedure} for neighborhood reconstruction.\\

\begin{algorithm}[H] 
	\KwData{An adjacency matrix $A \in \mathbb{R}^{n\times n}$ for the graph $G(V,E)$.} 
	\KwResult{For each node $v_i, 1\leq i\leq n$, an estimated set for neighborhood $\hat{\mathcal{N}}(v_i)$.} 
		1. For each node $v_i$, calculate the correlation $\langle A_{i}, A_{j} \rangle$ for all $j \neq i$ \;
		2. Sort the $\left\{ \langle A_{i}, A_{j} \rangle, j \in [n]\backslash\{i\} \right\}$ in a decreasing order, select the largest $k$ ones to form the estimated set $\hat{\mathcal{N}}(v_i)$  \; 
	\KwOut{$\hat{\mathcal{N}}(v_i)$, for all $i \in [n]$ } 		
	\caption{Correlation Thresholding for Neighborhood Reconstruction}
		\label{alg:corre-thres} 
\end{algorithm}
\medskip

The following lemma proves consistency of the above Algorithm~\ref{alg:corre-thres}. Note the computational complexity is $\mathcal{O}(n \cdot \min\{\log n, k\})$ for each node using quick-sort, with a total runtime $\mathcal{O}^*(n^2)$. 

\begin{lemma}[Consistency of Correlation Thresholding]
	\label{lma:corr-thres}
	Consider the Watts-Strogatz random graph ~${\sf WS}(n,k,\beta)$. Under  the reconstructable regime \uppercase\expandafter{\romannumeral4} (in Figure~\ref{fig:phase.tran}), that is,
	\begin{align}
		\label{eq:corr.thres}
		1-\beta \succ \sqrt{\frac{\log n}{k}} \vee \left( \frac{\log n}{n} \right)^{1/4},
	\end{align}
	correlation thresholding provides a consistent estimate of the neighborhood set $\mathcal{N}(v_i)$ w.h.p in the sense that
	\begin{align*}
		\lim_{n,k(n) \rightarrow \infty} \max_{i \in [n]} \frac{|\hat{\mathcal{N}}(v_i)	\triangle \mathcal{N}(v_i)|}{ |\mathcal{N}(v_i)|} = 0,
	\end{align*}
	where $ \triangle $ denotes the symmetric set difference. 
\end{lemma}
	One interesting question in small-world networks is to distinguish between strong ties (structural edges induced by the ring lattice structure) and weak ties (edges due to random connections). The above lemma addresses this question by providing a consistent estimate of the neighborhood set for each node.
	
	The condition under which consistency of correlation thresholding is ensured corresponds to the reconstructable region in Figure~\ref{fig:phase.tran}. One may ask if there is another algorithm that can provide a consistent estimate of the neighborhood set beyond region \uppercase\expandafter{\romannumeral4}. The answer is yes, and we will show in the following section that under the regime when $k$ is large (for instance, $k \succeq n^{\frac{15}{16}}$), indeed it is possible to slightly improve on Algorithm~\ref{alg:corre-thres}.

\subsection{Spectral Ordering}
\label{sec:spec-order}
Consider the following {\bf spectral ordering procedure}, which approximately reconstructs the ring lattice structure when $k$ is large, i.e., $k \succ n^{\frac{7}{8}}$.   \\

\begin{algorithm}[H] 
	\KwData{An adjacency matrix $A \in \mathbb{R}^{n\times n}$ for the graph $G(V,E)$.} 
	\KwResult{A ring embedding of the nodes $V$.} 
		1. Calculate top $3$ eigenvectors in the SVD $A = U \Sigma U^T$. Denote second and third eigenvectors as $u \in \mathbb{R}^{n}$ and $v \in \mathbb{R}^{n}$, respectively\; 
		2. For each node $i$ and the associated vector $A_{\cdot i} \in \mathbb{R}^n $, calculate the associated angle $\theta_i$ for vector $(u^T A_{\cdot i}, v^T A_{\cdot i})$\; 
	\KwOut{the sorted sequence $\{\theta_i \}_{i=1}^n$ and the corresponding ring embedding of the nodes. For each node $v_i$, $\hat{\mathcal{N}}(v_i)$ are the closest $k$ nodes in the ring embedding.} 		\caption{Spectral Reconstruction of Ring Structure}
		\label{alg:spectral} 
\end{algorithm}
\medskip

The following Lemma~\ref{lma:spectral} shows that when $k$ is large, Algorithm~\ref{alg:spectral} also provides consistent reconstruction of the ring lattice. Its computational complexity is $\mathcal{O}^*(n^2)$.

\begin{lemma}[Guarantee for Spectral Ordering]
	\label{lma:spectral}
	Consider the Watts-Strogatz graph ~${\sf WS}(n,k,\beta)$. Assume $k$ is large enough in the following sense:
	$$1> \varlimsup_{n,k(n)\rightarrow \infty} \frac{\log k}{\log n} \geq \varliminf_{n,k(n)\rightarrow \infty} \frac{\log k}{\log n} > \frac{7}{8}.$$ Under the regime 
	\begin{align}
		\label{eq:spec.ord}
		1-\beta \succ \frac{n^{3.5}}{k^4},
	\end{align}
	the spectral ordering provides consistent estimate of the neighborhood set $\mathcal{N}(v_i)$  w.h.p. in the sense that 
	\begin{align*}
		\lim_{n,k(n) \rightarrow \infty} \max_{i \in [n]} \frac{|\hat{\mathcal{N}}(v_i)	\triangle \mathcal{N}(v_i)|}{ |\mathcal{N}(v_i)|} = 0,
	\end{align*}
	where $ \triangle $ denotes the symmetric set difference.
\end{lemma}

	In Lemma~\ref{lma:spectral}, we can only prove consistency of spectral ordering under the technical condition that $k$ is large. We do not believe this is due to an artifact of the proof. Even though the structural matrix (the signal) has large eigenvalues, the eigen-gap is not large enough. The spectral ordering succeeds when the spectral gap stands out over the noise level, which implies that $k$ needs to be large enough. 
	
	Let us compare the region described in Equation~\eqref{eq:spec.ord} with the reconstructable region in Equation~\eqref{eq:corr.thres}. We observe that spectral ordering pushes slightly beyond the reconstructable region when $k \succ n^{\frac{15}{16}}$, as shown in Figure~\ref{fig:recon}.
	\begin{figure}[H] 
		\centering 
		\includegraphics[width=3in]{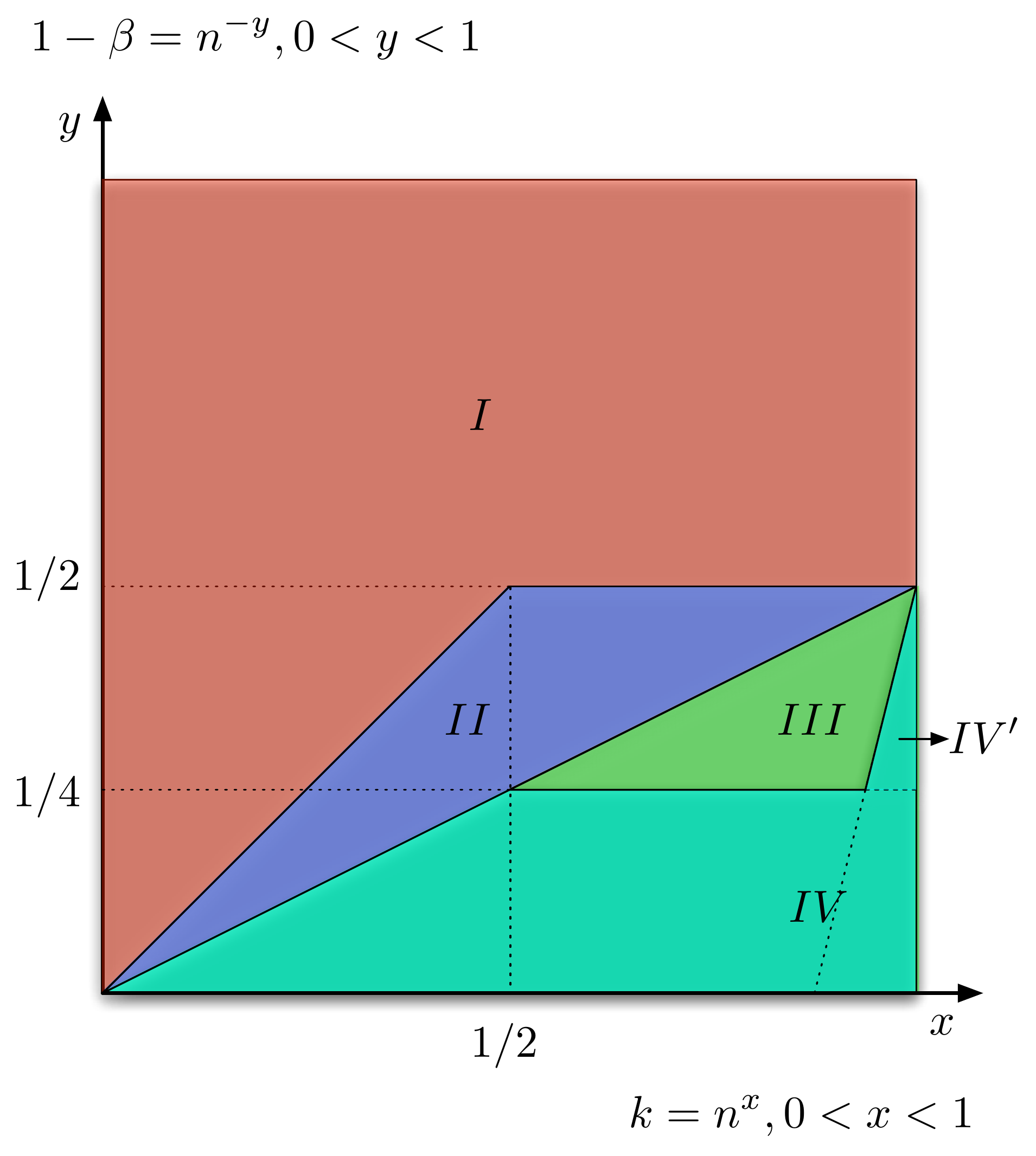} 
		\caption{Phase diagram for small-world networks: impossible region ({\color{red} red} region \uppercase\expandafter{\romannumeral1}), hard region ({\color{blue} blue} region \uppercase\expandafter{\romannumeral2}), easy region ({\color{green} green} region \uppercase\expandafter{\romannumeral3}), and reconstructable region ({\color{cyan} cyan} region \uppercase\expandafter{\romannumeral4} and \uppercase\expandafter{\romannumeral4}'). Compared to  Figure~\ref{fig:phase.tran}, the spectral ordering procedure extends the reconstructable region (\uppercase\expandafter{\romannumeral4})  when $k \succ n^{\frac{15}{16}}$ (\uppercase\expandafter{\romannumeral4}').}

		\label{fig:recon}
	\end{figure}

\section{Discussion}
\label{sec:discuss}

\paragraph{Reconstructable region}
We addressed the reconstruction problem via two distinct procedures, correlation thresholding and spectral ordering; however, whether there exists other computationally efficient algorithm that can significantly improve upon the current reconstructable region is still unknown. Designing new algorithms requires a deeper insight into the structure of the small-world model, and will probably shed light on better algorithms for mixed membership models.

\paragraph{Comparison to stochastic block model}
Recently, stochastic block models (SBM) have attracted considerable amount of attention from researchers in various fields. Community detection in stochastic block models focuses on recovering the hidden community information from the adjacency matrix that contains both noise and the latent permutation. The hidden community structure for classic SBM is illustrated in Figure~\ref{fig:sbm-sw} (the left one), as a block diagonal matrix. 
An interesting but theoretically more challenging extension to the classic SBM is the mixed membership SBM, where each node may simultaneously belong to several communities. The problem becomes more difficult when there are a growing number of communities and when each node belongs to several communities at the same time. Consider one easy case of the mix membership model, where the mix membership occurs only within neighborhood communities, as shown in the middle image of Figure~\ref{fig:sbm-sw}. The small-world network we are investigating in this paper can be seen as an extreme case (shown on the right-most figure) of this easy mixed membership SBM, where each node falls in effectively $k$ local clusters.
	\begin{figure}[H] 
		\centering 
		\includegraphics[width=5in]{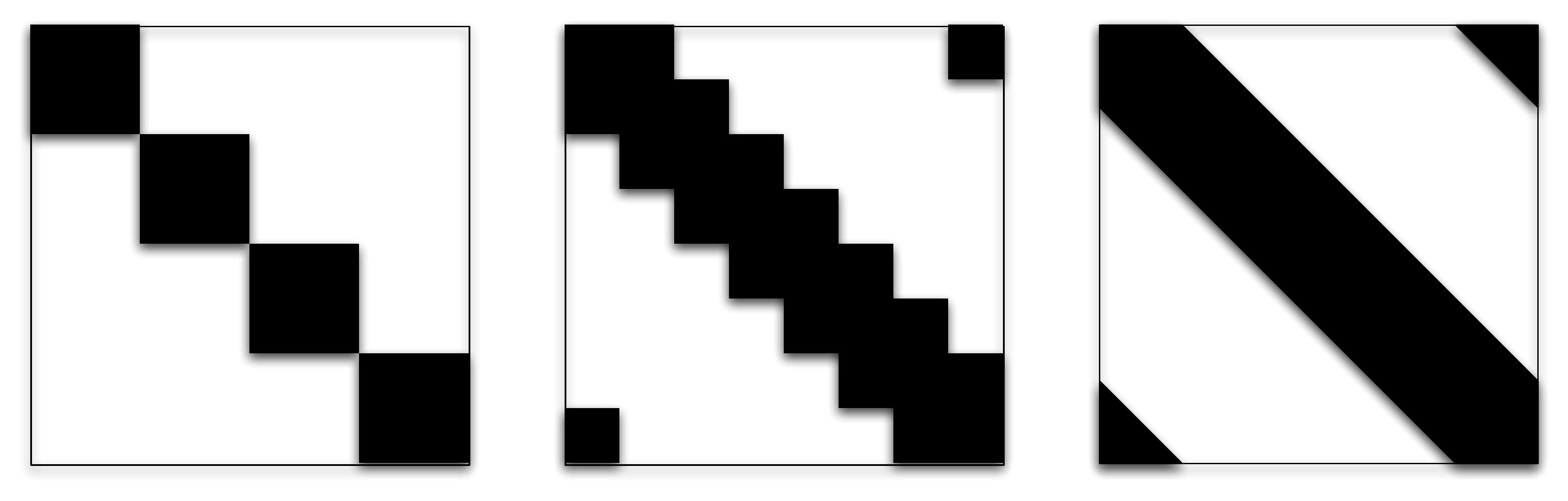} 
		\caption{The structural matrices for stochastic block model (left), mixed membership SBM (middle), and small-world model (right). The black location denotes the support of the structural matrix.}
		\label{fig:sbm-sw}
	\end{figure}
In the small-world networks, identifying the structural links and random links becomes challenging since that are many local clusters (in constrast to relative small number of communities in SBM). This multitude of local clusters makes it difficult to analyze the effect of the hidden permutation on the structural matrix. We view the current paper as an initial attempt at attacking this problem. %Whether reconstruction is possible beyond the regions we outlined is an interesting subject that requires further investigation.

\section{Technical Proofs}
\label{sec:proof}
\begin{proof}[Proof of Theorem~\ref{thm:low.bound}]

Denote the circulant matrix by $B$ (it is $B_\pi$ for any $\pi \in S_{n-1}$). The likelihood on $X \in \mathbb{R}^{n \times n}$ for WS model is
\begin{align*}
	\mathcal{L}_{n,k,\beta}(X|B) &= \exp\left\{ \log \frac{1 - \beta(1-\beta\frac{k}{n-1})}{\beta(1-\beta\frac{k}{n-1})} \cdot \langle X, B \rangle + \log \frac{\beta \frac{k}{n-1}}{1-\beta \frac{k}{n-1}} \cdot \langle X, J - I - B \rangle \right.\\
	&\quad  \quad \quad \quad \quad  \left. + nk \log (\beta(1-\beta\frac{k}{n-1})) + n(n-1-k) \log (1 - \beta \frac{k}{n-1}) \right\} \\
	&=\exp\left\{ \left( \log \frac{1 - \beta(1-\beta\frac{k}{n-1})}{\beta(1-\beta\frac{k}{n-1})} - \log \frac{\beta \frac{k}{n-1}}{1-\beta \frac{k}{n-1}} \right) \cdot \langle X, B \rangle + \log \frac{\beta \frac{k}{n-1}}{1-\beta \frac{k}{n-1}} \cdot \langle X, J - I \rangle  \right.\\
	& \quad  \quad \quad \quad \quad  \left. + nk \log (\beta(1-\beta\frac{k}{n-1})) + n(n-1-k) \log (1 - \beta \frac{k}{n-1}) \right\}.
\end{align*}
For the Erd\H{o}s-R\'{e}nyi model, the likelihood is 
\begin{align*}
	\mathcal{L}_{n,k}(X) = \exp\left\{ \log \frac{\frac{k}{n-1}}{1 - \frac{k}{n-1}} \cdot \langle X, J - I \rangle + n(n-1) \log (1- \frac{k}{n-1}) \right\} .
\end{align*}
The Kullback-Leibler divergence between this two model is expressed in the following
\begin{align}
	{\rm KL}(P_B|| P_0 ) & = \mathbb{E}_{X \sim P_B} \log\frac{P_B(X)}{P_{0}(X) } \nonumber \\
	& =  \mathbb{E}_{X \sim P_B} \left\{ - \left(\log \frac{\frac{k}{n-1}}{1 - \frac{k}{n-1}} - \log \frac{\beta \frac{k}{n-1}}{1-\beta \frac{k}{n-1}} \right) \cdot \langle X, J - I \rangle - n(n-1) \log (1- \frac{k}{n-1}) \right. \nonumber\\
	&\quad \left. + \left( \log \frac{1 - \beta(1-\beta\frac{k}{n-1})}{\beta(1-\beta\frac{k}{n-1})} - \log \frac{\beta \frac{k}{n-1}}{1-\beta \frac{k}{n-1}} \right) \cdot \langle X, B \rangle + nk \log (\beta(1-\beta\frac{k}{n-1})) + n(n-1-k) \log (1 - \beta \frac{k}{n-1}) \right\} \nonumber\\
	& = -\left(\log \frac{\frac{k}{n-1}}{1 - \frac{k}{n-1}} - \log \frac{\beta \frac{k}{n-1}}{1-\beta \frac{k}{n-1}} \right) \cdot  \left\langle (1-\beta)(1-\beta \frac{k}{n-1}) B + \beta\frac{k}{n-1} (J- I), J - I \right\rangle \nonumber\\
	& \quad + \left( \log \frac{1 - \beta(1-\beta\frac{k}{n-1})}{\beta(1-\beta\frac{k}{n-1})} - \log \frac{\beta \frac{k}{n-1}}{1-\beta \frac{k}{n-1}} \right) \cdot \left\langle (1-\beta)(1-\beta \frac{k}{n-1}) B + \beta\frac{k}{n-1} (J- I), B \right\rangle \nonumber\\
	& \quad - n(n-1) \log (1- \frac{k}{n-1}) + nk \log (\beta(1-\beta\frac{k}{n-1})) + n(n-1-k) \log (1 - \beta \frac{k}{n-1})\nonumber \\
	& = n(n-1) \log \frac{1 - \beta \frac{k}{n-1}}{ 1- \frac{k}{n-1}} - nk \log \frac{1}{\beta} -  \left[ \log \frac{1}{\beta} + \log \frac{1-\beta\frac{k}{n-1}}{1 - \frac{k}{n-1}}  \right] nk \left[1 - (1-\beta) \beta\frac{k}{n-1}\right] \nonumber \\
	& \quad \quad + \left[ \log \frac{1}{\beta} + \log \frac{1 - \beta(1-\beta \frac{k}{n-1})}{\beta \frac{k}{n-1}} \right] nk \left[1 - \beta(1-\beta \frac{k}{n-1})\right] \nonumber\\
	& = - \log \frac{1}{\beta} \cdot nk \left[ 1+\beta - \beta \frac{k}{n-1} \right] + \log \frac{1 - \beta \frac{k}{n-1}}{ 1- \frac{k}{n-1}} n\left[(n-1 - k) + (1-\beta) \beta\frac{k^2}{n-1} \right]\nonumber \\
	& \quad \quad \quad \quad \quad+ \log \frac{1 - \beta(1-\beta \frac{k}{n-1})}{\beta \frac{k}{n-1}} \cdot nk \left[1 - \beta(1-\beta \frac{k}{n-1})\right]. \label{eq:final}
\end{align}
Via the inequality $\log(1+x)<x$ for all $x>-1$, we can further simplify the above expression as
\begin{align}
	{\rm KL}(P_B|| P_0 ) & \leq nk(1-\beta)\left[ -\beta + \beta\frac{k}{n-1} + (1-\beta) \beta \frac{k^2}{n(n-1-k)} \right] + \frac{(1-\beta)(1-\beta\frac{k}{n-1})}{\beta \frac{k}{n-1}} nk \left[ (1-\beta) +\beta^2 \frac{k}{n-1} \right] \\
	& \leq nk(1-\beta) \left[(1-\beta)\beta \frac{k}{n-1} + (1-\beta) \beta \frac{k^2}{n(n-1-k)} \right] + \frac{(1-\beta)^2 (1-\beta\frac{k}{n-1})}{\beta} n(n-1) \leq C\cdot  n^2 (1-\beta)^2,
\end{align}
where $0<C<\frac{1}{2}\frac{k^2}{n(n-1)}+\frac{1}{\beta}$ is some universal constant (note we are interested in the case when $\beta$ is close to 1). 

Remark that when $k \preceq n^{1/2}$, the above bound can be further strengthened, in the following sense (recall equation~\eqref{eq:final})
\begin{align*}
	{\rm KL}(P_B|| P_0 ) & \leq nk(1-\beta)\left[ -\beta + \beta\frac{k}{n-1} + (1-\beta) \beta \frac{k^2}{n(n-1-k)} \right] + \log \frac{1 - \beta(1-\beta \frac{k}{n-1})}{\beta \frac{k}{n-1}} \cdot nk \left[1 - \beta(1-\beta \frac{k}{n-1})\right] \\
	& \leq \left\{ \log \frac{1 - \beta(1-\beta \frac{k}{n-1})}{\beta \frac{k}{n-1}} \cdot  \frac{1 - \beta(1-\beta \frac{k}{n-1})}{\beta \frac{k}{n-1}} \right\} \cdot k^2 \beta \frac{n}{n-1}.
\end{align*}
Denote $t:= \frac{1 - \beta(1-\beta \frac{k}{n-1})}{\beta \frac{k}{n-1}} = \frac{1-\beta}{\beta}\frac{n-1}{k} + \beta$. Thus we have
\begin{align}
\label{eq:kl.bd.2}
	{\rm KL}(P_B|| P_0 ) & \leq t \log t  \cdot k^2 \beta \frac{n}{n-1} .
\end{align}
Suppose for some constant $\alpha_*>0$, and $\alpha = \alpha_* \cdot \frac{1}{\beta} (1-\frac{1}{n})^2$, we have the following 
\begin{align}
& t \leq \alpha \frac{n\log \frac{n}{e}}{k^2} \cdot \frac{1}{\log \alpha \frac{n\log \frac{n}{e}}{k^2}} \label{eq:tlogt}\\
\text{and}~~~ &  t\log t \leq   \alpha \frac{n\log \frac{n}{e}}{k^2} \cdot \left( 1 - \frac{\log \log \alpha \frac{n\log \frac{n}{e}}{k^2}}{\log \alpha \frac{n\log n}{k^2}}\right) < \alpha \frac{n\log \frac{n}{e}}{k^2}  .
\end{align}
Plug in the expression for $t$ into \eqref{eq:tlogt}, if
\begin{align}
	\label{eq:key1}
\frac{1-\beta}{\beta} \leq  \alpha (1+\frac{1}{n-1}) \cdot \frac{\log \frac{n}{e}}{k}  \cdot \frac{1}{\log \alpha \frac{n\log \frac{n}{e}}{k^2}}- \frac{k}{n-1} \asymp \frac{\log n}{k} \frac{1}{\log  \frac{n\log \frac{n}{e}}{k^2}}
\end{align}
we have
$$
t \leq \alpha \frac{n\log \frac{n}{e}}{k^2} \cdot \frac{1}{\log \alpha \frac{n\log \frac{n}{e}}{k^2}}~~ \Rightarrow~~ t\log t < \alpha \frac{n\log \frac{n}{e}}{k^2} 
$$
which further implies (via equation~\eqref{eq:kl.bd.2})
\begin{align*}
	\frac{1}{(n-1)!} \sum_{\pi \in S_{n-1}} {\rm KL}(P_{B_\pi} || P_0) \leq  t \log t  \cdot k^2 \beta \frac{n}{n-1} \leq \alpha_* \cdot \log (n-1)!.
\end{align*}
Recalling Equation~\eqref{eq:final}, if
\begin{align}
	\label{eq:key2}
	1-\beta \leq \sqrt{ \frac{\alpha_*}{C} \cdot \frac{(n-1) \log \frac{n}{e} }{  n^2}}  \asymp \sqrt{\frac{\log n}{n}}
\end{align}
we have
\begin{align*}
	\frac{1}{(n-1)!} \sum_{\pi \in S_{n-1}} {\rm KL}(P_{B_\pi} || P_0) \leq  n^2 (1-\beta)^2 \leq \alpha_*\cdot \log (n-1)!.
\end{align*}
 
We invoke the following Lemma on minimax error through Kullbak-Leibler divergence.
\begin{lemma}[\cite{tsybakov2009introduction}, Proposition 2.3]
	\label{lma:fano.type}
	Let $P_0$, $P_1$,\ldots, $P_M$ be probability measures on $(\mathcal{X},\mathcal{A})$ satisfying
	\begin{align}
		\frac{1}{M} \sum_{j=1}^M {\rm KL}(P_j || P_0 )  \leq \alpha \cdot \log M
	\end{align}
	with $0<\alpha<\frac{1}{8}$. Then for any $\psi:\mathcal{X} \rightarrow [M+1]$
	\begin{align*}
		\max ~\left\{ P_0(\psi \neq 0), \frac{1}{M} \sum_{j=1}^M P_j(\psi \neq j)  \right\} \geq \frac{\sqrt{M}}{\sqrt{M}+1} \left( 1 - 2\alpha - \sqrt{\frac{2\alpha}{\log M}}\right).
	\end{align*}
\end{lemma}
 
Collecting Equations~\eqref{eq:key1} and \eqref{eq:key2}, if either one of the conditions in Equations~\eqref{eq:key1} and \eqref{eq:key2} holds, we have
 \begin{align}
	 \label{eq:cond}
 \frac{1}{(n-1)!} \sum_{\pi \in S_{n-1}} {\rm KL}(P_{B_\pi} || P_0) \leq \alpha_* \cdot \log (n-1)!.
 \end{align}
 Putting things together, if$$1-\beta \prec \sqrt{\frac{\log n}{n}} \vee \frac{\log n}{k}, $$ we have that Equation~\eqref{eq:cond} hold. Applying  Lemma~\ref{lma:fano.type}, we complete the proof
$$
\varliminf_{n\rightarrow \infty} ~ \min_{\phi} ~\max~\left\{ P_0(\phi \neq 0),~ \frac{1}{(n-1)!} \sum_{i=1}^{(n-1)!} P_i(\phi \neq i) \right\} \geq \varliminf_{n\rightarrow \infty} ~  \frac{\sqrt{(n-1)!}}{1+\sqrt{(n-1)!}} \left( 1- 2\alpha - \sqrt{\frac{2\alpha}{\log (n-1)!}} \right) = 1 -2\alpha .
$$
\end{proof}

\bigskip

\begin{proof}[Proof of Lemma~\ref{lma:maxlik.test}]
	Let us state the well-known Bernstein's inequality (\cite{boucheron2013concentration}, Theorem 2.10), which will be used in the proof of this lemma. 
	\begin{lemma}[Bernstein's inequality]
		\label{lma:bern.ineq}
		Let $X_1,\ldots, X_n$ be independent bounded real-valued random variables. Assume that there exist positive numbers $v$ and $c$ such that 
		\begin{align*}
		& \sum_{i=1}^n \mbb{E} [X_i^2] \leq v,\\
		& X_i \leq 3c, \forall 1\leq i\leq n~~ a.s.
		\end{align*}
		then we have, for all $t>0$,  
		\begin{align}
		\mbb{P} \left(  \sum_{i=1}^n (X_i - \mbb{E} X_i) \geq \sqrt{2vt}+ct   \right) \leq e^{-t}.
		\end{align}
	\end{lemma}
	
	First, let us consider the case when the adjacency matrix $A$ is generated from the Erd\H{o}s-R\'{e}nyi random graph ${\sf ER}(n,\frac{k}{n-1})$. Recall Bernstein's inequality Lemma~\ref{lma:bern.ineq}, for any $P_
	\pi$ with $\pi \in S_{n-1}$, we know $\langle P_\pi B P_\pi^T, A \rangle$ has the same distribution as $\langle  B, A \rangle$. Thus
	\begin{align*}
		\langle P_\pi B P_\pi^T, A \rangle &\overset{\mathrm{in~law}}{=\joinrel=} \langle  B, A \rangle =  2 \sum_{i>j}  A_{ij} B_{ij} \\
		& = 2 \sum_{i>j} \mbb{E}[A_{ij}] B_{ij}+2 \sum_{i>j}  (A_{ij} -\mbb{E}[A_{ij}]) B_{ij} \\
		& \leq \frac{k}{n-1} nk +2 \sqrt{ \frac{k}{n-1} nk t} + \frac{2}{3}t
	\end{align*}
	with probability at least $1 - \exp(-t)$. Here the last step is through Bernstein's inequality. There are $nk/2$ non-zero $B_{i,j}, i>j$, and it is clear that $A_{ij} \sim {\sf Bernoulli}(\frac{k}{n-1})$,
	$2 \sum_{i>j} \mbb{E}[A_{ij}] B_{ij} = nk \frac{k}{n-1}$. Thus we can take $c = \frac{1}{3}$ and $$
	v = \sum_{i<j} \mbb{E}[A_{ij} B_{ij}]^2 = \sum_{i<j} \mbb{E}[A_{ij}]^2  B_{ij} =\frac{nk}{2} \frac{k}{n-1}$$
	in Lemma~\ref{lma:bern.ineq}. 
	Via the union bound, take $t = \log n!$, we have
	\begin{align*}
		\max_{P_\pi}~ 	\langle P_\pi B P_\pi^T, A \rangle  \leq \frac{k}{n-1} nk + 2\sqrt{\frac{k}{n-1} nk  \cdot \log n!} + \frac{2}{3} \cdot \log n!
	\end{align*}
	with probability at least $1 - (n-1)! \exp(- \log n!) = 1 - \frac{1}{n}$. 
	
	Alternatively, suppose $A$ is from the small-world rewiring model ${\sf WS}(n,k,\beta)$, with permutation being the identity $\pi = e$. 
	With probability at least $1 - \exp(-\log n) = 1- \frac{1}{n}$,
		\begin{align*}
		\max_{P_\pi}~ 	\langle P_\pi B P_\pi^T, A \rangle &\geq 		\langle B,  A \rangle \\
		& = \langle B , \mbb{E} [A]  \rangle + \langle B , A - \mbb{E} [A]  \rangle \\
		& \geq (1-\beta+ \beta^2\frac{k}{n-1} )  nk - \sqrt{nk \cdot \log n} 
		\end{align*}
		where the last step is from Hoeffding's inequality: it is clear that for location $(i,j)$ when $B_{ij} \neq 0$, 
		$$
		\mathbb{E}[A_{ij}] = 1-\beta+\beta^2 \frac{k}{n-1},
		$$
		and $0\leq A_{ij}  \leq 1$ almost surely. 
	
		Therefore if there exist a threshold $T>0$ such that
		\begin{align}
		\label{eq:threshold}
		(1-\beta+ \beta^2\frac{k}{n-1} )  nk - \sqrt{nk \cdot \log n}  > T >   \frac{k}{n-1} nk + 2\sqrt{\frac{k}{n-1} nk  \cdot \log n!} + \frac{2}{3} \cdot \log n!
		\end{align}
		we have that
		$$
			\lim_{n, k(n) \rightarrow \infty} ~\max~\left\{ P_0(\phi_1 \neq 0),~ \frac{1}{(n-1)!} \sum_{i=1}^{(n-1)!} P_i(\phi_1 \neq 1) \right\}  \leq 	\lim_{n, k(n) \rightarrow \infty} ~\frac{1}{n} = 0.
		$$

	The detailed calculation of Equation~\eqref{eq:threshold} yields that the test succeeds with high probability whenever
	\begin{align*}
		1-\beta \succeq \sqrt{\frac{\log n}{n}} \vee \frac{\log n}{k}.
	\end{align*}
\end{proof}

\begin{proof}[Proof of Lemma~\ref{lma:spec.test}]
	 Under the rewiring model (Watts-Strogatz model) ${\sf WS}(n,k,\beta)$ with permutation $P_\pi$
	\begin{align*}
		P_\pi A P_\pi^T = (1- \beta)(1-\beta\frac{k}{n-1})\cdot B + \beta\frac{k}{n-1} \cdot (J-I) + Z 
	\end{align*}
	where $J = 1 1^T \in \mathbb{R}^{n\times n}$, $B$ is the ring structured signal matrix defined in Equation~\eqref{eq:structure.mat}.  We denote in short $A = \mathbb{E} A + Z$ as this signal and the noise part, and
	\begin{equation*}
		B_{ij} = \left\{
		\begin{array}{lc}
			1 & \text{if}~ 0<|i-j| \leq \frac{k}{2} \mod n-1-\frac{k}{2} \\
			0 & \text{elsewhere}
		\end{array}
		\right.
	\end{equation*}
	$B$ is a circulant matrix, whose spectrum is highly structured, and 
	$Z$ is a zero-mean noise random matrix. 
	
	We first study the random fluctuation part, $Z = A - \mathbb{E} A$. 
	Let us bound the expectation $\mbb{E} \| A - \mbb{E} A \|$ as a starting step, for any adjacency matrix $A \in \mbb{R}^{n \times n}$ using the symmetrization trick. Denote $A'\sim A$ as the independent copy of A sharing the same distribution. Take $E, G \in \mbb{R}^{n \times n}$ as random symmetric Rademacher and Gaussian matrices with entries $E_{ij}$, $G_{ij}$ being, respectively, independent Rademacher and Gaussian. Denoting $A \circ B$ as matrix Hadamard product, we have
	\begin{align*}
		\mbb{E} \| A - \mbb{E} A\| & = \mbb{E} \sup_{ \| v \|_{\ell_2} =1} \langle (A - \mbb{E}A) v, v \rangle = \mbb{E} \sup_{ \| v \|_{\ell_2} =1} \langle (A - \mbb{E}_{A'}A') v, v \rangle \\
		& \leq \mbb{E}_{A} \mbb{E}_{A'} \sup_{ \| v \|_{\ell_2} =1} \langle (A - A') v, v \rangle =  \mbb{E}_{E} \mbb{E}_{A} \mbb{E}_{A'} \sup_{ \| v \|_{\ell_2} =1} \langle [E \circ (A - A')] v, v \rangle  \\
		& \leq \mbb{E}_{A} \mbb{E}_{E}  \sup_{ \| v \|_{\ell_2} =1} \langle [E \circ A]  v, v \rangle +  \mbb{E}_{A'} \mbb{E}_{E}  \sup_{ \| v \|_{\ell_2} =1} \langle [-E \circ A'] v, v \rangle  \\
		& = 2 \mbb{E}_{A} \mbb{E}_{E}  \sup_{ \| v \|_{\ell_2} =1} \langle [E \circ A]  v, v \rangle  \leq \frac{2}{\sqrt{2/\pi}}  \cdot \mbb{E}_{A} \mbb{E}_{E}  \sup_{ \| v \|_{\ell_2} =1} \langle [ \mbb{E}_G[|G|] \circ E \circ A]  v, v  \rangle \\
		& \leq \sqrt{\frac{\pi}{2}} \cdot \mbb{E}_{A} \mbb{E}_{E}  \mbb{E}_G  \sup_{ \| v \|_{\ell_2} =1} \langle [|G| \circ E \circ A]  v, v  \rangle = \sqrt{\frac{\pi}{2}} \cdot \mbb{E}_{A} \mbb{E}_G  \sup_{ \| v \|_{\ell_2} =1} \langle [G \circ A]  v, v  \rangle \\
		& =  \sqrt{\frac{\pi}{2}} \cdot \mbb{E}_{A} \left( \mbb{E}_G \| G \circ A \| \right).
	\end{align*}
	Recall the following Lemma from \cite{bandeira2014sharp}.
	\begin{lemma}[\cite{bandeira2014sharp}, Theorem 1.1] Let $X$ be the $n \times n$ symmetric random matrix with $X = G \circ A$, where $G_{ij}, i<j$ are i.i.d. $N(0,1)$ and $A_{ij}$ are given scalars. Then
	$$
	\mbb{E}_{G} \| X \| \leq \max_{i} \sqrt{\sum_j A^2_{ij}} + \max_{ij} |A_{ij}| \cdot \sqrt{\log n}.
	$$
	\end{lemma}
	Thus via Jensen's inequality and the above Lemma, we can continue
	\begin{align*}
		\mbb{E} \| A - \mbb{E} A\| &\leq  \sqrt{\frac{\pi}{2}} \cdot \mbb{E}_{A} \left( \mbb{E}_G \| G \circ A \| \right) \\
		& \precsim  \mbb{E}_{A} \max_{i} \sqrt{\sum_j A^2_{ij}} + \max_{ij} |A_{ij}| \cdot \sqrt{\log n} \\
		& \leq \sqrt{ \mbb{E}_{A} \max_{i} \sum_j A^2_{ij} } + \sqrt{\log n}\\
		& \leq \sqrt{k +C_1 2\sqrt{k\log n} + C_2 \log n} + \sqrt{\log n} \asymp \sqrt{k} \vee \sqrt{\log n},
	\end{align*}
	where the last step uses Bernstein inequality Lemma~\ref{lma:bern.ineq}. Moving from expectation $\mbb{E} \| A - \mbb{E} A\| $ to concentration on $\| A - \mbb{E} A\| $ is through Talagrand's concentration inequality (see, \cite{talagrand1996new} and \cite{tao2012topics} Theorem 2.1.13), since $\| \cdot \|$ is $1-$Lipschitz convex function in our case (and the entries are bounded), thus with probability at least $1 - \frac{1}{n}$,
	$$
	\| A - \mbb{E} A\|  \leq   \mbb{E} \| A - \mbb{E} A\| + C \cdot \sqrt{\log n} \asymp \sqrt{k} \vee \sqrt{\log n}.
	$$

	Now let us study the structural signal part. Matrix $B$ is of the form circulant matrix, the associated polynomial is $$ f(x) = (x + x^{n-k/2}) \cdot \frac{x^{k/2}-1}{x-1}. $$ The eigen-structure is: collect for all $j = 0, 1, ..., n/2$ $$ (\cos 0,\cos \frac{2\pi j}{n},\cos \frac{2\pi 2j}{n},\ldots, \cos \frac{2\pi nj}{n} ) $$ and $$ (\sin 0,\cos \frac{2\pi j}{n},\sin \frac{2\pi 2j}{n},\ldots, \sin \frac{2\pi nj}{n} ) $$ and the corresponding eigenvalue is $$\lambda_j = f(w_j) = 2\sum_{i=1}^{k/2} \cos \left( i \frac{2\pi j}{n} \right).$$

	Let us first assume $\frac{k}{n}\leq \frac{1}{2}$, thus $\lambda$ is the second largest eigenvalue $$ \lambda = 2\sum_{i=1}^{k/2} \cos \left( i \frac{2\pi}{n} \right) = \frac{2 \sin\frac{k \pi}{2n}}{\sin\frac{\pi}{n}} \cos\frac{(k+2)\pi}{2n} \asymp k .$$ 

Using Weyl's interlacing inequality, if there exist a $T>0$ such that
$$
\lambda_2(A_{\sf WS}) \geq \lambda_2(\mbb{E}[A_{\sf WS}]) - \| Z \| > T > \|Z' \| > \lambda_2(A_{\sf ER}) ,
$$
where
\begin{align*}
 \lambda_2(M) - \| Z \|  &\geq 	(1- \beta)(1-\beta\frac{k}{n-1}) \lambda  -\sqrt{k} \vee \sqrt{\log n}, \\
 \| Z'\| &\leq \sqrt{k} \vee \sqrt{\log n},
\end{align*}
then we have
$$
	\lim_{n, k(n) \rightarrow \infty} ~\max~\left\{ P_0(\phi_2 \neq 0),~ \frac{1}{(n-1)!} \sum_{i=1}^{(n-1)!} P_i(\phi_2 \neq 1) \right\} = 0.
$$
Therefore, we have the condition for which the second eigenvalue test succeeds:
\begin{align*}
	(1- \beta)(1-\beta\frac{k}{n-1}) \lambda & >  \sqrt{k} \vee \sqrt{\log n}\\
	(1- \beta)(1-\beta\frac{k}{n-1}) & > \frac{\sqrt{k \log n} \vee \log n }{\frac{2 \sin\frac{k \pi}{2n}}{\sin\frac{\pi}{n}} \cos\frac{(k+2)\pi}{2n}} \asymp \sqrt{ \frac{1}{k} } \vee  \frac{\sqrt{\log n}}{k}.
\end{align*}
\end{proof}

\bigskip

\begin{proof}[Proof of Lemma~\ref{lma:corr-thres}]
Take any two vectors $A_{i \cdot}, A_{j \cdot}$ that are two rows of the adjacency matrix. Denote the $i,j$-th rows have distance $|\pi^{-1}(i) - \pi^{-1}(j)|_{\rm ring} = x$. This is equivalent to saying that the Hamming distance of the corresponding signal vectors satisfies ${\rm H}(B_{i \cdot}, B_{j \cdot}) = 2x-2$. Therefore the union of signal nodes for $i,j$-th row is of cardinality $|S_i \cup S_j| = k+x-1$, common signal nodes is of cardinality $|S_i \cap S_j| = k - x +1$, unique signal is of cardinality $|S_i \triangle S_j| =  2x-2$, and $|S_i^c \cap S_j^c| = n - k - x - 1$.  The signal nodes is $1$ with probability $p =  1- \beta (1-\beta\frac{k}{n-1})$, non signal is $1$ with probability $q = \beta\frac{k}{n-1}$, and we have
\begin{align*}
	\langle A_{i \cdot}, A_{j \cdot} \rangle = \sum_{l \in S_{i} \cap S_{j}} A_{il} A_{jl} + \sum_{l \in S_{i} \triangle S_{j}} A_{il} A_{jl}  + \sum_{l \in S_i^c \cap S_j^c} A_{il} A_{jl}.
\end{align*} 
Observe as long as $l \neq i,j$, $A_{il}$ and $A_{jl}$ are independent, and $\left\{ A_{il}A_{jl}, l \in [n]\backslash \{i,j\} \right\}$ are independent of each other.  

Let us bound each term via Bernstein's inequality  Lemma~\ref{lma:bern.ineq},
\begin{align*}
	\sum_{l \in S_{i} \cap S_{j}} A_{il} A_{jl} \in p^2 |S_{i} \cap S_{j}| \pm \left(\sqrt{2p^2 |S_{i} \cap S_{j}| t} + \frac{1}{3} t \right) \\
	\sum_{l \in S_{i} \triangle S_{j}} A_{il} A_{jl} \in pq |S_{i} \triangle S_{j}| \pm \left(\sqrt{2pq |S_{i} \triangle S_{j}| t} + \frac{1}{3} t \right) \\
	\sum_{l \in S_{i}^c \cap S_{j}^c} A_{il} A_{jl} \in q^2 |S_{i}^c \cap S_{j}^c| \pm \left(\sqrt{2q^2 |S_{i}^c \cap S_{j}^c| t} + \frac{1}{3} t \right)
\end{align*}
with probability at least $1 - 6\exp(-t)$. We take $t = (2+\epsilon) \log n$ for any $\epsilon>0$, such that with probability at least $1 - C n^{-\epsilon}$, the above bound holds for all pairs $(i,j)$.

Thus for all $|\pi^{-1}(i) - \pi^{-1}(j)|_{\rm ring} > k$ pairs, 
\begin{align*}
	\langle A_{i \cdot}, A_{j \cdot} \rangle \leq 2k pq + (n-2k-2) q^2 + \left( \sqrt{4kpqt} + \sqrt{2(n-2k-2)q^2t} + t \right),
\end{align*}
for $|\pi^{-1}(i) - \pi^{-1}(j)|_{\rm ring}\leq x$ pairs
\begin{align*}
	\langle A_{i \cdot}, A_{j \cdot} \rangle \geq (k - x +1)p^2 + (2x-2) pq + (n-k-x-1) q^2 - \left(\sqrt{2(k - x +1)p^2t} + \sqrt{2(2x-2) pqt} + \sqrt{2 (n-k-x-1)q^2t}  + t\right).
\end{align*}

Thus, with $t = (2+\epsilon) \log n$, $p = 1 - \beta (1-\beta\frac{k}{n-1})$ and $q = \beta\frac{k}{n-1}$, if $x<x_0$ with
$$\frac{x_0}{k} := 1 - C_1 \sqrt{\frac{\log n}{k}} \frac{1}{1-\beta} - C_2  \sqrt{\frac{\log n}{n}} \frac{1}{(1-\beta)^2} ,$$ we have
\begin{align*}
	 (k-x+1)(p-q)^2 & \geq 2t + (2\sqrt{2}+1) \left( \sqrt{kp^2}  + \sqrt{nq^2} \right) \sqrt{2t} \\
	 & \geq 2t + \left( \sqrt{2kpq} + \sqrt{(n-2k-2)q^2} + \sqrt{(k - x +1)p^2} + \sqrt{(2x-2) pq} + \sqrt{ (n-k-x-1)q^2} \right) \sqrt{2t},
\end{align*}
which further implies,
\begin{align*}
	 \min_{j: |\pi^{-1}(i) - \pi^{-1}(j)|_{\rm ring} \leq x_0}~~ \langle A_{i \cdot}, A_{j \cdot} \rangle &  \geq \max_{ j \notin \mathcal{N}(v_i)}~~\langle A_{i \cdot}, A_{j \cdot} \rangle, \forall i \\
	  \max_{i \in [n]} \frac{|\hat{\mathcal{N}}(v_i)	\triangle \mathcal{N}(v_i)|}{ |\mathcal{N}(v_i)|} & \leq \frac{k-x_0}{k} =   C_1 \sqrt{\frac{\log n}{k}} \frac{1}{1-\beta} + C_2  \sqrt{\frac{\log n}{n}} \frac{1}{(1-\beta)^2}.
\end{align*}
Therefore we can reconstruct the neighborhood consistently, under the condition
$$
1-\beta \succ \sqrt{\frac{\log n}{k}} \vee \left( \frac{\log n}{n} \right)^{1/4}.
$$ 
\end{proof}

\bigskip
\begin{proof}[Proof of Lemma~\ref{lma:spectral}]
Since eigen-structure is not affected by permutation, we will work under the case when the true permutation is identity. We work under a mild technical assumption that we have two independent observation of the adjacency matrix, one used for calculated the eigen-vector, the other used for projection. Note this is only a technical assumption for simplicity, and does not affect the theoretical result.
Recall that $A = M + Z$, where $M = (1- \beta)(1-\beta\frac{k}{n-1})\cdot B + \beta\frac{k}{n-1} \cdot (J-I) $ is the signal matrix. Denote the eigenvectors of $M$ to be $U \in \mathbb{R}^{n \times n}$, and eigenvectors of $A$ to be $\hat{U}\in \mathbb{R}^{n \times n}$. Classic Davis-Kahan perturbation bound informs us that 
	\begin{align*}
		\| \hat{U}_{\cdot 2} - U_{\cdot 2} \| \leq \frac{ \| Z \| }{ \Delta \lambda - \| Z \|}, ~~ \| \hat{U}_{\cdot 3} - U_{\cdot 3} \| \leq \frac{ \| Z \| }{ \Delta \lambda - \| Z \|},
	\end{align*}
where the spectral gap $\Delta \lambda$ of $M$ is
\begin{align*}
\Delta \lambda &:= (1- \beta)(1-\beta\frac{k}{n-1})\cdot( \lambda_2 - \lambda_3 )=(1- \beta)(1-\beta\frac{k}{n-1}) \cdot\left[ 2\sum_{i=1}^{k/2} \cos \left( i \frac{2\pi}{n} \right) - 2\sum_{i=1}^{k/2} \cos \left( i \frac{2\pi \cdot 2}{n} \right)  \right] \\
 & = (1- \beta)(1-\beta\frac{k}{n-1}) \left[ \frac{2 \sin\frac{k \pi}{2n}}{\sin\frac{\pi}{n}} \cos\frac{(k+2)\pi}{2n} - \frac{2 \sin\frac{k \pi}{n}}{\sin\frac{2\pi}{n}} \cos\frac{(k+2)\pi}{n} \right]  \asymp (1- \beta)(1-\beta\frac{k}{n-1}) \frac{k^3}{n^2}.
\end{align*}
From the proof of Lemma~\ref{lma:spec.test}, we know with high probability
$$
\| Z \| \preceq \sqrt{k} \vee \sqrt{\log n}.
$$

We denote
	\begin{align*}
		\tan \hat{\theta}_i  = \frac{\hat{U}_{\cdot 3}^T A_{\cdot i}}{\hat{U}_{\cdot 2}^T A_{\cdot i}},
	\end{align*}
	and
	\begin{align*}
		\tan \theta_i = \frac{\langle U_{\cdot 3}, M_{\cdot i} \rangle}{\langle U_{\cdot 2}, M_{\cdot i} \rangle}  =  \frac{\frac{\lambda_2}{\sqrt{n}} \sin \frac{(i-1)2\pi}{n}   }{ \frac{\lambda_2}{\sqrt{n}} \cos \frac{(i-1)2\pi}{n} } = \tan \frac{(i-1)2\pi}{n}.
	\end{align*}
	Observe that
	\begin{align}
		\label{eq:tan}
		\tan \hat{\theta}_i &= \frac{\langle \hat{U}_{\cdot 3}, A_{\cdot i} \rangle }{\langle \hat{U}_{\cdot 2}, A_{\cdot i} \rangle},
	\end{align}
	and for both the denominator and numerator, we have the bound
	\begin{align*}
		\langle \hat{U}_{\cdot 2}, A_{\cdot i} \rangle &=\langle (\hat{U}_{\cdot 2} - U_{\cdot 2}) + U_{\cdot 2}, M_{\cdot i} + Z_{\cdot i} \rangle  \\
		&= \langle U_{\cdot 2}, M_{\cdot i} \rangle + \langle \hat{U}_{\cdot 2} - U_{\cdot 2}, M_{\cdot i} \rangle + \langle \hat{U}_{\cdot 2} , Z_{\cdot i} \rangle \\
		& \leq U_{\cdot 2}, M_{\cdot i} \rangle + \| \hat{U}_{\cdot 2} - U_{\cdot 2} \| \| M_{\cdot i} \| +| \langle \hat{U}_{\cdot 2} , Z_{\cdot i} \rangle |.
	\end{align*}
	Thus we know
	\begin{align}
	& \max\left\{ | \langle \hat{U}_{\cdot 2}, A_{\cdot i} \rangle  - \langle U_{\cdot 2}, M_{\cdot i} \rangle |, | \langle \hat{U}_{\cdot 3}, A_{\cdot i} \rangle  - \langle U_{\cdot 3}, M_{\cdot i} \rangle | \right\}	\\
	&\leq  \max\left\{  \| \hat{U}_{\cdot 2} - U_{\cdot 2} \| \| M_{\cdot i} \| +| \langle \hat{U}_{\cdot 2} , Z_{\cdot i} \rangle |,  \| \hat{U}_{\cdot 3} - U_{\cdot 3} \| \| M_{\cdot i} \| +| \langle \hat{U}_{\cdot 3} , Z_{\cdot i} \rangle |\right\} \\
		&\leq  \frac{ \sqrt{k} \vee \sqrt{\log n}}{\lambda_2 - \lambda_3 - \sqrt{k} \vee \sqrt{\log n} }\cdot  \sqrt{k} (1-\beta) + \sqrt{\log n} 
	\end{align}
	where the last line follows from the definition of principal angle and Davis-Kahan bound and Hoeffding's inequality for $ \langle \hat{U}_{\cdot 2} , Z_{\cdot i} \rangle$. Proceeding with Equation~\eqref{eq:tan}, without loss of generality, assume $0\leq \frac{(i-1)2\pi}{n} \leq \frac{\pi}{2}$, for $0\leq \frac{(i-1)2\pi}{n} \leq \frac{\pi}{4}$ 
	\begin{align}
		\tan \hat{\theta}_i 
		& \leq \frac{\frac{\lambda_2}{\sqrt{n}} \sin \frac{(i-1)2\pi}{n}  + \frac{ \sqrt{k} \vee \sqrt{\log n} }{\lambda_2 - \lambda_3 - \sqrt{k} \vee \sqrt{\log n} }\cdot \sqrt{k} (1-\beta) + \sqrt{\log n}  }{ \frac{\lambda_2}{\sqrt{n}} \cos \frac{(i-1)2\pi}{n}  - \frac{ \sqrt{k} \vee \sqrt{\log n} }{\lambda_2 - \lambda_3 - \sqrt{k} \vee \sqrt{\log n} }\cdot \sqrt{k} (1-\beta)  - \sqrt{\log n} } \\
		& = \frac{ \sin \frac{(i-1)2\pi}{n} + \delta}{  \cos \frac{(i-1)2\pi}{n} - \delta},\quad \text{with $\delta = \frac{\sqrt{n}}{\lambda_2} \cdot \left\{  \frac{ \sqrt{k} \vee \sqrt{\log n} }{\lambda_2 - \lambda_3 - \sqrt{k} \vee \sqrt{\log n} }\cdot \sqrt{k} (1-\beta) + \sqrt{\log n} \right\} $} \\
		\tan \hat{\theta}_i  & \geq  \frac{ \sin \frac{(i-1)2\pi}{n} - \delta}{  \cos \frac{(i-1)2\pi}{n} + \delta},
	\end{align}
	with similar bounds for $\cot \hat{\theta}_i$ with $\frac{\pi}{4} \leq \frac{(i-1)2\pi}{n} \leq \frac{\pi}{2}$.  
	Here the stochastic eror is bounded in the sense $\delta \asymp \frac{n^{2.5}}{k^3}\frac{1}{1-\beta} \rightarrow 0$.
	From the above equation, we have
	\begin{align*}
		| \hat{\theta}_i - \theta_i | &\leq \min\{ | \tan \hat{\theta}_i - \tan \theta_i|, |\cot\hat{\theta}_i - \cot \theta_i | \} \\
		& \leq \min\left\{ \frac{\delta(1+\tan \theta_i)}{\cos \theta_i - \delta}, \frac{\delta(1+\cot \theta_i)}{\sin \theta_i - \delta} \right\}  \leq \frac{2\delta}{\frac{1}{\sqrt{2}}-\delta}.
	\end{align*}
	For any $i$,  we have the bound on the stochastic error $| \hat{\theta}_i - \theta_i | < C \cdot \delta \asymp  \frac{n^{2.5}}{k^3}\frac{1}{1-\beta} $. And for all $j \in \mathcal{N}(i)$ in the neighborhood, the support is $| \theta_j - \theta_i | \leq \frac{2\pi k }{n}$. Fix any $i$, for any $j \notin \mathcal{N}(v_i)$,
	$$
	\min_{j \notin \mathcal{N}(v_i) }~~|\hat{\theta}_{j} - \hat{\theta}_i| \geq  \frac{2\pi k }{n} - C \delta >  (\frac{2\pi k }{n} - C' \delta) + C \delta \geq \max_{|j - i|<k-C' n\delta} |\hat{\theta}_{j} - \hat{\theta}_i|,
	$$
	with $C \leq 4\sqrt{2}$ and $C'>2C$.
	Therefore, the following bound on symmetric set difference holds
	$$
	 \max_{i \in [n]} \frac{|\hat{\mathcal{N}}(v_i)	\triangle \mathcal{N}(v_i)|}{ |\mathcal{N}(v_i)|} \leq  \frac{ C' \cdot n\delta}{k}  \leq \frac{C' \cdot n \cdot \frac{n^{2.5}}{k^3}\frac{1}{1-\beta}}{k}.
	$$

	In summary under the condition
	$$1 - \beta \succ \frac{n^{3.5}}{k^4},$$
	one can recover the neighborhood consistently w.h.p. in the sense
	$$
	\lim_{n,k(n) \rightarrow \infty} \max_{i \in [n]} \frac{|\hat{\mathcal{N}}(v_i)	\triangle \mathcal{N}(v_i)|}{ |\mathcal{N}(v_i)|} = 0.
	$$ 
	
\end{proof}
\section*{acknowledgement}
The authors thank Elchanan Mossel for many helpful discussions and
suggestions for improving the paper.

%%%%%%%%%%%%%%%%%%%%%%%%%%%%%%%%%%%%%%%%%%%%%%%%%%%%%%%
\bibliographystyle{plainnat} 
\bibliography{bibfile}

\end{document}